\documentclass[10pt]{article}
\usepackage[a4paper,centering,scale=0.75]{geometry}
\usepackage{mathtools} % loads amsmath automatically
\usepackage{amsthm,amssymb} % not needed if amsart is used
\usepackage{mathrsfs}

\newtheorem{prop}{Proposition}[section]
\newtheorem{thm}[prop]{Theorem}

\theoremstyle{remark}

\theoremstyle{definition}
\newtheorem{hyp}[prop]{Assumption}

\newcommand{\sC}{\mathsf{C}}
\newcommand{\dom}{\mathsf{D}}
\newcommand{\ep}{\varepsilon}
\newcommand{\cF}{\mathscr{F}}
\newcommand{\cL}{\mathscr{L}}

\renewcommand{\geq}{\geqslant}
\renewcommand{\leq}{\leqslant}

\numberwithin{equation}{section}

\DeclarePairedDelimiter\abs{\lvert}{\rvert}
\DeclarePairedDelimiter\norm{\lVert}{\rVert}
\DeclarePairedDelimiterX\ip[2]{\langle}{\rangle}{#1,#2}

\newcommand{\embed}{\hookrightarrow}
\newcommand{\longto}{\longrightarrow}

\begin{document}
\title{On the differentiability of solutions to \\ singularly
  perturbed SPDEs}
\author{Carlo Marinelli\thanks{Department of Mathematics, University
    College London, Gower Street, London WC1E 6BT, United
    Kingdom. URL: \texttt{http://goo.gl/4GKJP}}}
\date{\normalsize December 30, 2020}

\maketitle

\begin{abstract}
  We consider semilinear stochastic evolution equations on Hilbert
  spaces with multiplicative Wiener noise and linear drift term of the
  type $A + \varepsilon G$, with $A$ and $G$ maximal monotone
  operators and $\varepsilon$ a ``small'' parameter, and study the
  differentiability of mild solutions with respect to
  $\varepsilon$. The operator $G$ can be a singular perturbation of
  $A$, in the sense that its domain can be strictly contained in the
  domain of $A$.
\end{abstract}

% -----------------------------------------------------------------------

\section{Introduction}
Let us consider a stochastic evolution equation of the type
\begin{equation}
  \label{eq:ep}
  du_\ep + (A+\ep G)u_\ep\,dt =
  f(u_\ep)\,dt + B(u_\ep)\,dW,
  \qquad u_\ep(0) = u_0,
\end{equation}
posed on a Hilbert space $H$, where $A$ and $G$ are linear maximal
monotone operators on $H$ such that the closure of $A+\ep G$ (denoted
by the same symbol for simplicity), with $\ep$ a positive parameter,
is also maximal monotone for sufficiently small values of $\ep$.
Moreover, $W$ is a cylindrical Wiener process, and the random
time-dependent coefficients $f$, $B$ satisfy suitable measurability
and Lipschitz continuity conditions. 
Then \eqref{eq:ep} admits a unique mild solution $u_\ep$ with
continuous paths for every positive $\ep$ close to or equal to
zero. Denoting the mild solution to \eqref{eq:ep} with $\ep=0$ by $u$,
if $A+\ep G$ converges to $A$ in the strong resolvent sense as $\ep
\to 0$, then $u_\ep$ converges to $u$ in a rather strong topology
(implying, in particular, the uniform convergence in probability on
compact time intervals). 

Our main interest is about cases where the domain of $G$ is included
in the domain of $A$ and $A$ is ``hyperbolic'' (in the terminology of
Kato, cf.~\cite{Kato:sing}, i.e. the semigroup $S_A$ generated by $-A$
is not holomorphic).
The aim is then to study the differentiability of arbitrary order of
the map $\ep \mapsto u_\ep$, under suitable assumptions on the
interplay between the semigroup $S_A$ and $G$ and regularity
assumptions on the initial datum and the coefficients of
\eqref{eq:ep}. This allows to construct series expansions in $\ep$ of
functionals of $u_\ep$ ``centered'' around the corresponding
functionals of $u$, appealing to arguments based on Taylor's formula,
composition of (Fr\'echet) differentiable maps, and composition of
formal series. Such asymptotic expansions have been obtained in
\cite{cm:epsd1} in the simpler case where $f$ and $B$ are random
time-dependent maps that do not explicitly depend on the unknown. The
results of \cite{cm:epsd1} have been obtained under the assumption
that the semigroups generated by $-A$ and $-G$ commute, which implies
that $S_{A+\ep G}=S_A S_{\ep G}$ (with obvious meaning of the
notation), and using a Taylor-type formula for $S_{\ep G}$ (see
\cite[Proposition~1.1.6]{BuBe}). Here we proceed in a completely
different way, which essentially consists in differentiating
\eqref{eq:ep} with respect to the parameter $\ep$. This method seems
both clearer and more powerful, at least in the sense that, even
assuming that $S_{A+\ep G}=S_A S_{\ep G}$, the technique of expanding
$S_{\ep G}$ in a series of Taylor type does not appear to be helpful
in the case of equations with drift and diffusion coefficients
depending on the unknown.
It would not be difficult, in our approach, to allow $u_0$, $f$ and
$B$ to depend on $\ep$ too, at the cost of little more than heavier
notation. We also do not assume that $A$ and $G$ are (negative)
generators of commuting semigroups, even though our alternative weaker
hypotheses are admittedly still quite strong and most likely not easy
to check. On the other hand, the differentiability of $\ep \mapsto
u_\ep$, even in the linear deterministic setting where both $f$ and
$B$ are zero, is difficult to obtain already at order two
(cf.~\cite[pp.~506-ff.]{Kato}). It seems hence unlikely that
asymptotic expansions as $\ep \to 0$ can be obtained without rather
heavy assumptions.

We remark that, even though we discuss in this paper only the
differentiability of $u_\ep$ with respect to $\ep$, all results about
series expansions of functionals of $u_\ep$ around the corresponding
functionals of $u$ obtained in \cite{cm:epsd1} immediately extend to
the more general situation considered here. In particular, one can
extend the results on the parabolic regularization of a suitably
extended Musiela's SPDE obtained in \emph{op.~cit.}, where the
volatilities and the drift term are random time-dependent maps that do
not depend on the forward curve, to the ``full'' Musiela SPDE where the
volatilities, hence the drift term, may indeed depend on the forward
curve.

The content of the remaining text is organized as follows: in
section~\ref{sec:prel}, after stating the main assumptions on $A$ and
$G$, we obtain some auxiliary results of analytic nature on the
behavior of the semigroup with negative generator $A$ and $A+\ep G$ on
the domain of powers of $G$. Moreover, we recall results on existence,
uniqueness, and continuous dependence on the coefficients for
stochastic evolution equations on Hilbert spaces.
In section~\ref{sec:fd} we formally differentiate \eqref{eq:ep} with
respect to $\ep$ and establish the existence and uniqueness of
solutions to the stochastic evolution equations thus obtained, as well
as their continuous dependence on the parameter $\ep$. In the final
sections \ref{sec:d1} and \ref{sec:dh} we prove the differentiability
of the map $\ep \mapsto u_\ep$, showing that its derivative of order
$k$ coincides with the mild solution to the equations obtained by
formally differentiating equation \eqref{eq:ep} $k$ times with respect
to $\ep$.

\medskip\par\noindent%
\textbf{Acknowledgments.} The author gratefully acknowledges the
hospitality of the Interdisziplin\"ares Zentrum f\"ur Komplexe
Systeme at the University of Bonn, where this paper was written.

\section{Preliminaries and auxiliary results}
\label{sec:prel}
Let us begin with some conventions and notations. All random
quantities will be defined on a fixed probability space
$(\Omega,\cF,\mathbb{P})$ endowed with a filtration $(\cF_t)_{t \in
  [0,T]}$ satisfying the so-called usual conditions. A fixed
cylindrical Wiener process on a separable Hilbert space $U$ will be
denoted by $W$.  A Hilbert space $H$, with scalar product
$\ip{\cdot}{\cdot}$ and norm $\norm{\cdot}$, will be fixed from now
on.  Let $E$ be a Banach space. The (quasi-)normed space of adapted,
continuous $E$-valued processes $X$ such that
\[
\norm[\big]{X}_{\sC^p(E)} := \norm[\big]{X}_{L^p(\Omega;C([0,T];E))} < \infty,
\]
with $p>0$, is denoted by $\sC^p(E)$, or just by $\sC^p$ if
$E=H$. This is a Banach spaces for $p \geq 1$, and a quasi-Banach
space for $p<1$.
The Banach space of multilinear maps from $H^k$ to $E$, with $k$ a
positive integer, will be denoted by $\cL_k(H;E)$. If $E$ is a Hilbert
space, we shall write $\cL^2(U;E)$ to mean the Hilbert space of
Hilbert-Schmidt operators from $U$ to $E$.

\medskip

Let $A$ and $G$ be linear maximal monotone operators on $H$. We shall
denote the strongly continuous semigroup of contractions on $H$
generated by $-A$ by $S_A$ (and analogously for any other maximal
monotone operator). We recall that $G$ is a closed operator, hence so
is every power $G^k$, with $k$ positive integer, and the domain
$\dom(G^k)$ of $G^k$, endowed with the scalar product
\[
\ip[\big]{\phi}{\psi}_{\dom(G^k)} := \ip{\phi}{\psi} +
\ip{G\phi}{G\psi} + \cdots + \ip{G^k\phi}{G^k\psi},
\]
is a Hilbert space.

The following assumption will be in force throughout.
\begin{hyp}
  \label{h:1}
  There exists $\ep_0$ such that the closure of $A+\ep G$ is maximal
  monotone for every $\ep \in \mathopen]0,\ep_0]$. Moreover, there
  exist an integer $m \geq 1$ and constants $\alpha_1,\ldots,\alpha_{m+1}$
  such that, for every $k \in \{1,\ldots,m+1\}$,
  \[
  \norm[\big]{G^k S_A(t)\phi} \leq e^{\alpha_k t} \norm[\big]{G^k \phi}
  \qquad \forall \phi \in \dom(G^k), \; t \geq 0.
  \]
\end{hyp}
\noindent The assumption is satisfied, for instance, if $S_A$ and
$S_G$ commute. Moreover, it immediately implies that there exists a
constant $\alpha$ such that
\[
  \norm[\big]{S_A(t)\phi}_{\dom(G^m)} \leq e^{\alpha t}
  \norm[\big]{\phi}_{\dom(G^m)}
  \qquad \forall \phi \in \dom(G^m), \; t \geq 0.
\]
Note that $A+\ep G$ is closable because it is monotone, so the
assumption about it is that, for every $\ep \in \mathopen]0,\ep_0]$,
there exists $\lambda>0$ such that the image of $\lambda + A +\ep G$
is dense in $H$ . We shall abuse notation in the following writing
$A+\ep G$ to mean its closure, and we shall assume, for simplicity,
that $\ep_0=1$.

\subsection{Analytic preliminaries}
Assumption~\ref{h:1} implies analogous estimates for $S_{A+\ep G}$.
\begin{prop}
  \label{prop:seh}
  One has, for every $k \in \{1,\ldots,m+1\}$,
  \[
  \norm[\big]{G^k S_{A+\ep G}(t)\phi} \leq e^{\alpha_k t} \norm[\big]{G^k \phi}
  \qquad \forall \phi \in \dom(G^k), \; t \geq 0.
  \]
\end{prop}
\begin{proof}
  Let $k \in \{1,\ldots,m+1\}$ and $\phi \in \dom(G^k)$. One has, by the
  Trotter product formula (see, e.g., \cite[p.~227]{EnNa}),
  \begin{equation}
    \label{eq:tro}
    S_{A+ \ep G}(t) \phi = \lim_{n \to \infty} \bigl( S_A(t/n)
    S_{\ep G}(t/n) \bigr)^n \phi,
  \end{equation}
  where, for any natural $n \geq 1$,
  \begin{align*}
    \norm[\big]{G^k S_A(t/n)
    S_{\ep G}(t/n) \phi} 
    &\leq e^{\alpha_k t/n} \norm[\big]{G^k S_{\ep G}(t/n) \phi}\\
    &= e^{\alpha_k t/n} \norm[\big]{S_{\ep G}(t/n) G^k\phi}\\
    &\leq e^{\alpha_k t/n} \norm[\big]{G^k\phi}.
  \end{align*}
  Let us show that if 
  \[
    \norm[\big]{G^k \bigl(S_A(t/n) S_{\ep G}(t/n)\bigr)^j \phi} \leq
    e^{\alpha_k tj/n} \norm[\big]{G^k\phi}
  \]
  (which has just been proved with $j=1$) holds for a natural number $j<n$,
  then the same inequality holds with $j$ replaced by $j+1$. In fact, 
  \begin{align*}
  &\norm[\big]{G^k \bigl(S_A(t/n) S_{\ep G}(t/n)\bigr)^{j+1} \phi}\\
  &\hspace{3em} = \norm[\big]{G^k S_A(t/n) S_{\ep G}(t/n)
    \bigl( S_A(t/n) S_{\ep G}(t/n)\bigr)^j \phi}\\
  &\hspace{3em} \leq e^{\alpha_k t/n} \norm[\big]{G^k
    \bigl( S_A(t/n) S_{\ep G}(t/n)\bigr)^j \phi}\\
  &\hspace{3em} \leq e^{\alpha_k (j+1)t/n} \norm[\big]{G^k \phi}.
  \end{align*}
  In particular, one has
  \[
    \norm[\big]{G^k \bigl(S_A(t/n) S_{\ep G}(t/n)\bigr)^n \phi} \leq
    e^{\alpha_k t} \norm[\big]{G^k \phi},
  \]
  which implies that there exists $\zeta \in H$ such that, passing to
  a subsequence if necessary,
  \[
    G^k \bigl(S_A(t/n) S_{\ep G}(t/n)\bigr)^n \phi \to \zeta
  \]
  weakly in $H$. The closedness of $G^k$ (in the strong-weak sense)
  and \eqref{eq:tro} imply that $\zeta = G^kS_{A+\ep G}(t)\phi$, and
  weak lower semicontinuity of the norm yields
  \[
    \norm[\big]{G^kS_{A+\ep G}(t)\phi} \leq e^{\alpha_k t} \norm[\big]{G^k \phi}.
    \qedhere
  \]
\end{proof}

Let us introduce the resolvents of $A$ and $A+\ep G$ defined, for
$\lambda>0$, as
\[
  R_\lambda := (\lambda + A)^{-1}, \qquad
  R_\lambda(\ep) := (\lambda + A + \ep G)^{-1},
\]
respectively. By $R_\lambda(0)$ we shall mean $R_\lambda$.

\begin{prop}
  \label{prop:rc}
  For any $\ep \in [0,1]$ and $\lambda>\alpha_1$, $R_\lambda(\ep+h)x$
  converges to $R_\lambda(\ep)x$ as $h \to 0$ for all $x \in H$.
\end{prop}
\begin{proof}
  Let $\phi \in \dom(G)$. One has, by assumption, $\norm{GS_A(t)\phi}
  \leq e^{\alpha_1 t} \norm{G\phi}$, which implies, recalling that $G$
  is closed, that, for every $\lambda>\alpha_1$,
  \[
  G R_\lambda \phi = \int_0^\infty e^{-\lambda t} GS_A(t)\phi\,dt,
  \]
  hence
  \[
  \norm{G R_\lambda \phi} \leq \int_0^\infty e^{-\lambda t}
  \norm{GS_A(t)\phi}\,dt \leq \frac{1}{\lambda-\alpha_1} \norm{G\phi}.
  \]
  The second resolvent identity yields
  \[
  R_\lambda(\ep)\phi - R_\lambda\phi = \ep R_\lambda(\ep) GR_\lambda \phi,
  \]
  where
  \[
  \norm{R_\lambda(\ep) GR_\lambda \phi} \leq
  \norm{GR_\lambda \phi} \leq \frac{1}{\lambda-\alpha_1}
  \norm{G\phi}.
  \]
  Therefore $R_\lambda(\ep) - R_\lambda$ converges pointwise to zero
  on $\dom(G)$, which is dense in $H$, and is bounded in $\cL(H)$
  uniformly with respect to $\ep \in [0,1]$. From this it follows that
  $R_\lambda(\ep)$ converges to $R_\lambda$ pointwise on $H$.
  Taking Proposition~\ref{prop:seh} into account, the proof that
  $R_\lambda(\ep+h)$ converges to $R_\lambda(\ep)$ pointwise in $H$ as
  $h \to 0$ is completely analogous.
\end{proof}

\begin{prop}
  \label{prop:gen}
  The restrictions of the semigroups $S_A$ and $S_{A + \ep G}$ on
  $\dom(G^k)$ are strongly continuous quasi-contraction semigroups
  thereon for every $k \in \{1,\ldots,m+1\}$.
\end{prop}
\begin{proof}
  The family of operators $S_A$ and $S_{A+\ep G}$ are
  quasi-contractive endomorphisms of $\dom(G^k)$ for every
  $k \in \{1,\ldots,m+1\}$ by Assumption~\ref{h:1} and
  Proposition~\ref{prop:seh}, respectively.
  In the following we write $S$ to mean either $S_A$ or
  $S_{A + \ep G}$, as the argument is identical. Let us first prove
  weak continuity, i.e. that $\phi(S(t)h) \to \phi(h)$ for every
  $\phi \in \dom(G^k)'$, the dual of $\dom(G^k)$. By the Riesz
  representation theorem, for every $\phi \in \dom(G^k)'$ there exists
  $\ell_\phi \in \dom(G^k)$ such that
  \begin{equation}
    \label{eq:phili}
    \phi\colon h \mapsto \ip[\big]{\ell_\phi}{h}_{\dom(G^k)} =
    \ip[\big]{\ell_\phi}{h} + \ip[\big]{G\ell_\phi}{Gh} + \cdots +
    \ip[\big]{G^k\ell_\phi}{G^kh}.
  \end{equation}
  Let $j \in \{1,\ldots,k\}$ and $h \in \dom(G^j)$ be fixed. Thanks to
  Assumption~\ref{h:1} and Proposition~\ref{prop:seh}, $G^jS(t)h$ is
  bounded in $H$, hence there exists $\zeta \in H$ and a sequence
  $(t_n)$ converging to zero such that $G^jS(t_n)h$ converges weakly
  to $\zeta$ in $H$ as $n \to \infty$. Since $G^j$ is closed, hence
  also strongly-weakly closed, one has $\zeta = G^jh$. By uniqueness
  of the limit we conclude that $G^jS(t)h$ converges weakly to $G^jh$
  in $H$ as $t \to 0$. This immediately implies, in view of
  \eqref{eq:phili}, that $S(t)h \to h$ weakly in $\dom(G^k)$ as
  $t \to 0$. The proof is concluded recalling that weak continuity for
  semigroups implies strong continuity (see, e.g.,
  \cite[p.~40]{EnNa}).
\end{proof}

The convergence of resolvent of Proposition~\ref{prop:rc} continues to
hold also on $\dom(G^m)$, as we now show.
\begin{prop}
  \label{prop:core}
  Let $\ep \in [0,1]$ and $k \in \{1,\ldots,m\}$. The resolvent
  $R_\lambda(\ep+h)$ converges to $R_\lambda(\ep)$ pointwise on
  $\dom(G^k)$ as $h \to 0$ for every
  $\lambda>\alpha_k \vee \alpha_{k+1}$.
\end{prop}
\begin{proof}
  The proof is analogous to that of Proposition~\ref{prop:rc}: for any
  $\phi \in \dom(G^{k+1})$ one infers, by Assumption~\ref{h:1}, that
  \[
    \norm{G^k R_\lambda \phi} \leq \frac{1}{\lambda-\alpha_k}
    \norm{G^k\phi}.
  \]
  The second resolvent identity implies
  \[
    G^kR_\lambda(\ep)\phi - G^kR_\lambda\phi
    = \ep G^kR_\lambda(\ep) GR_\lambda \phi,
  \]
  where, by Proposition~\ref{prop:seh},
  \[
    \norm{G^kR_\lambda(\ep) GR_\lambda \phi} \leq
    \frac{1}{\lambda-\alpha_k} \norm{G^{k+1}R_\lambda \phi} \leq
    \frac{1}{(\lambda-\alpha_k)(\lambda-\alpha_{k+1})}
    \norm{G^{k+1} \phi}.
  \]
  This implies that $R_\lambda(\ep) - R_\lambda$ converges pointwise
  to zero on $\dom(G^{k+1})$, which is dense in $\dom(G^k)$, and is
  bounded in $\cL(\dom(G^k))$ uniformly with respect to
  $\ep \in [0,1]$. Therefore $R_\lambda(\ep)$ converges to $R_\lambda$
  pointwise on $\dom(G^k)$ as $\ep \to 0$.
  Using Proposition~\ref{prop:seh} instead of Assumption~\ref{h:1}
  allows one to establish, by the same argument, that
  $R_\lambda(\ep+h)$ converges to $R_\lambda(\ep)$ pointwise in
  $\dom(G^k)$ as $h \to 0$ for every $\ep \in \mathopen]0,1]$.
\end{proof}

\subsection{Existence, uniqueness and convergence of mild solutions}
We assume throughout that the maps
\[
  f: \Omega \times [0,T] \times H \longto H, \qquad
  B: \Omega \times [0,T] \times H \longto \cL^2(U;H)
\]
are progressively measurable (in the sense that they are measurable
with respect to the product $\sigma$-algebra
$\mathscr{R} \otimes \mathscr{B}(H)$, where $\mathscr{R}$ stands for
the $\sigma$-algebra of progressively measurable subsets of
$\Omega \times [0,T]$) and they satisfy the Lipschitz continuity
condition
\[
  \norm[\big]{f(\omega,t,x) - f(\omega,t,y)} 
  + \norm[\big]{B(\omega,t,x) - B(\omega,t,y)}_{\cL^2(U;H)}
  \lesssim \norm{x-y}
\]
for all $(\omega,t) \in \Omega \times [0,T]$, with implicit constant
independent of $\omega$ and $t$.
It is well known that if, for a $p>0$,
$u_0 \in L^p(\Omega,\mathscr{F}_0;H)$ and there exists $a \in H$ such
that
\begin{equation}
  \label{eq:faba}
  f(a) \in L^p(\Omega;L^1(0,T;H)), \qquad
  B(a) \in L^p(\Omega;L^2(0,T;\cL^2(U;H))),
\end{equation}
then the stochastic evolution equation on $H$
\begin{equation}
  \label{eq:a}
  du + Au\,dt = f(u)\,dt + B(u)\,dW, \qquad u(0)=u_0,
\end{equation}
admits a unique mild solution $u \in \sC^p$ and the solution map
$u_0 \mapsto u$ is Lipschitz continuous from $L^p(\Omega;H)$ to
$\sC^p$ (see, e.g., \cite{DP:K,cm:SIMA18}).

An important tool for the developments in the following sections is a
more precise continuous dependence result for the mild solution to
\eqref{eq:a} on its data. In particular, let us consider the family of
equations indexed by $n \in \mathbb{N}$
\begin{equation}
  \label{eq:un}
  du_n + A_nu_n\,dt = f_n(u_n)\,dt + B_n(u_n)\,dW,
  \qquad u_n(0)=u_{0n},
\end{equation}
where $A_n$ is a linear maximal monotone operator on $H$, $u_{0n} \in
L^0(\Omega,\mathscr{F}_0;H)$, and the maps
\[
  f_n: \Omega \times [0,T] \times H \longto H, \qquad
  B_n: \Omega \times [0,T] \times H \longto \cL^2(U;H)
\]
satisfy the same measurability conditions satisfied by $f$ and $B$,
respectively, and are Lipschitz continuous in their third argument,
uniformly over $\Omega \times [0,T]$ as well as with respect to $n$,
i.e. there exist positive constants $N_1$ and $N_2$, not depending on
$\omega$, $t$ and $n$, such that
\begin{gather*}
  \norm[\big]{f_n(\omega,t,x) - f_n(\omega,t,y)} \leq N_1 \norm{x-y},\\
  \norm[\big]{B_n(\omega,t,x) - B_n(\omega,t,y)}_{\cL^2(U;H)}
  \leq N_2 \norm{x-y}
\end{gather*}
for all $(\omega,t,n) \in \Omega \times [0,T] \times
\mathbb{N}$. Therefore, for any $p>0$, if $f_n$ and $B_n$ satisfy
condition \eqref{eq:faba} and
$u_{0n} \in L^p(\Omega,\mathscr{F}_0;H)$, there exists a unique mild
solution $u_n \in \sC^p$ to \eqref{eq:un}.

\begin{thm}
  \label{thm:str}
  Let $p>0$. Assume that $f$, $B$ as well as $f_n$ and $B_n$ satisfy
  condition \eqref{eq:faba} for every $n \in \mathbb{N}$. Furthermore,
  assume that, as $n \to \infty$, $A_n \to A$ in the strong resolvent
  sense, $f_n(\omega,t,\cdot)$ and $B_n(\omega,t,\cdot)$ converge to
  $f(\omega,t,\cdot)$ and $B(\omega,t,\cdot)$ pointwise in $H$ and
  $\cL^2(U;H)$, respectively, for all $(\omega,t) \in \Omega \times
  [0,T]$, and $u_{0n} \to u_0$ in $L^p(\Omega,\mathscr{F}_0;H)$. Then
  $u_n \to u$ in $\sC^p$.
\end{thm}
The proof is entirely similar to that of \cite[Theorem~2.4]{cm:JFA13}
and hence omitted (cf. also \cite{KvN2} where, however, $A_n$ as well
as $A$ have to be negative generators of holomorphic semigroups).

We are now in the position to state and prove the well-posedness and
convergence results for \eqref{eq:ep} needed in the following
sections. To this purpose, we shall say that $f$ and $B$ satisfy
condition $\mathrm{H}(m,p)$, $m \in \mathbb{N}$, $p>0$, if
\[
  f: \Omega \times [0,T] \times H \longto \dom(G^m), \qquad
  B: \Omega \times [0,T] \times H \longto \cL^2(U;\dom(G^m))
\]
are progressively measurable and they fulfill the Lipschitz continuity
condition
\begin{equation}
  \label{eq:lipa}
  \norm[\big]{f(\omega,t,x) - f(\omega,t,y)}_{\dom(G^m)}
  + \norm[\big]{B(\omega,t,x) - B(\omega,t,y)}_{\cL^2(U;\dom(G^m))}
  \lesssim \norm{x-y},
\end{equation}
with implicit constant independent of $\omega$ and $t$, and there
exists $a \in H$ such that \eqref{eq:faba} holds with $H$ replaced by
$\dom(G^{m})$.

\begin{prop}
  \label{prop:wpc}
  Let $p>0$ and $m \in \mathbb{N}$. Assume that $f$ and $B$ satisfy
  assumption $\mathrm{H}(m,p)$ and $u_0 \in
  L^p(\Omega,\mathscr{F}_0;\dom(G^m))$. Then equation \eqref{eq:ep}
  admits a unique mild solution $u_\ep \in \sC^p(\dom(G^m))$ for every
  $\ep \in [0,1]$. Moreover, $u_{\ep+h}$ converges to $u_\ep$ in
  $\sC^p(\dom(G^m))$ as $h \to 0$ for every $\ep \in [0,1]$.
\end{prop}
\begin{proof}
  Well-posedness is a consequence of Proposition~\ref{prop:gen}, upon
  observing that \eqref{eq:lipa} holds also when the norm on the
  right-hand side is replaced by the norm of $\dom(G^m)$. 
  The convergence of $u_{\ep+h}$ to $u_\ep$ in $\sC^p(\dom(G^m))$ as
  $h \to 0$ follows by Proposition~\ref{prop:core} and
  Theorem~\ref{thm:str}.
\end{proof}

\section{Equations for formal derivatives}
\label{sec:fd}
The purpose of this section is to establish well-posedness in
appropriate spaces of the linear stochastic evolution equations
obtained by \emph{formally} differentiating \eqref{eq:ep} with respect
to the parameter $\ep$. The formal $n$-th derivative of
$\delta \mapsto u_\delta$ at the point $\ep$ will be denoted by
$u_\ep^k$. Derivatives of $f$ and $B$ are always meant with
respect to their third argument. First order formal differentiation
yields
\begin{equation}
  \label{eq:uu}
  du^1_\ep + (A + \ep G)u^1_\ep\,dt + Gu_\ep
  = f'(u_\ep)u^1_\ep\,dt + B'(u_\ep)u^1_\ep\,dW,
  \qquad u^1_\ep(0)=0.
\end{equation}
In order to write the equations obtained by higher-order formal
derivatives, we observe that, for any function $\ep \mapsto g_\ep \in
C^\infty([0,1];\dom(G))$, one has,
\[
  \bigl((A+\ep G) g_\ep\bigr)^{(n)} =
  \sum_{j=0}^\infty \binom{n}{j} (A+\ep G)^{(j)}
  g_\ep^{(n-j)} = (A+\ep G)g_\ep^{(n)} +
  nGg_\ep^{(n-1)},
\]
as well as
\begin{align*}
  \bigl( f(g_\ep) \bigr)^{(n)} &= f'(g_\ep)
  g_\ep^{(n)} + r(f,g;n-1,\ep),\\
    \bigl( B(g_\ep) \bigr)^{(n)} &= B'(g_\ep)
  g_\ep^{(n)} + r(B,g;n-1,\ep),
\end{align*}
where $r(B,g;n-1,\ep)$ is a finite linear combination of terms
of the type
\[
  B^{(j)}(g_\ep)\bigl(
  g_\ep^{(n_1)},\ldots,g_\ep^{(n_j)}\bigr),
\]
with $j \in \{2,\ldots,n\}$, $n_1+\cdots+n_j=n$,
$n_1,\ldots,n_j \geq 1$.\footnote{These terms can be written in a more
  explicit way by the formula for the derivatives of higher order of a
  composite function (see, e.g., \cite[p.~272]{Bog:TVS}), but it is
  not important for our purposes.} In particular, note that
$r(B,g;n-1,\ep)$ depends only on terms of order at most $n-1$
and does not involve $G$.  Completely similar considerations clearly
hold also for $r(f,g;n-1,\ep)$.
Therefore one has, for every integer $n \geq 1$,
\begin{equation}
  \label{eq:uen}
  \begin{split}
  &du^n_\ep + (A+\ep G)u^n_\ep\,dt
    + nGu^{n-1}_\ep\,dt\\
  &\hspace{3em} = f'(u_\ep)u^n_\ep\,dt
    + \varphi_n(u_\ep)\,dt %
    + B'(u_\ep)u^n_\ep\,dW + \Phi_n(u_\ep)\,dW,\\
  &u^n_\ep(0)=0,
  \end{split}
\end{equation}
where $\varphi_n(u_\ep)$ and $\Phi_n(u_\ep)$ are obtained by
$r(\cdot,\cdot;n-1,\ep)$ in the obvious way.

Assuming that $\mathrm{H}(m,p)$ holds for a fixed $m \geq 1$, we are
going to show that, given any $n \in \{1,\cdots,m\}$, \eqref{eq:uen}
admits a unique mild solution $u^n \in \sC^p(\dom(G^{m-n}))$ under
suitable differentiability assumptions on $f$ and $B$.
Let us begin with the equation for first-order formal
derivatives. From now on we shall make the
\begin{hyp}
\label{h:d1}
  The maps
  \begin{align*}
    f(\omega,t,\cdot) \colon H
    &\longto \dom(G^{m-1}),\\
    B(\omega,t,\cdot) \colon H
    &\longto \cL_j(H;\cL^2(U;\dom(G^{m-1})))
  \end{align*}
  are of class $C^1$ for every $(\omega,t) \in \Omega \times [0,T]$.
\end{hyp}
The derivatives of $f$ and $B$ in the sense of this assumption will be
denoted by $f'$ and $B'$, respectively.  Note that the Lipschitz
continuity of $f$ and $B$ in \eqref{eq:lipa} implies that $f$ and $B$ in
this assumption are in fact of class $C^1_b$, not just $C^1$. More
explicitly, the uniform Lipschitz continuity of $B(\omega,t)$ from $H$
to $\cL^2(U;\dom(G^m))$ implies its uniform Lipschitz continuity, with
the same Lipschitz constant, also with codomain
$\cL^2(U;\dom(G^{m-1}))$, as $\dom(G^m)$ is clearly contractively
embedded in $\dom(G^{m-1})$. Therefore the norm in
$\cL_j(H;\cL^2(U;\dom(G^{m-1})))$ of $B'$ is bounded by its Lipschitz
constant, uniformly with respect to
$(\omega,t) \in \Omega \times [0,T]$. Completely similar
considerations hold for $f$.

\begin{prop}
  \label{prop:ufo}
  Let $p>0$, $u_0 \in L^p(\Omega,\mathscr{F}_0;H)$ and assumption
  $\mathrm{H}(m,p)$ be satisfied. Then equation \eqref{eq:uu} admits a
  unique mild solution $u^1_\ep \in \sC^p(\dom(G^{m-1}))$ for all
  $\ep \in [0,1]$.
\end{prop}
\begin{proof}
  By Proposition~\ref{prop:gen} $S_{A + \ep G}$ is a strongly
  continuous semigroup of quasi-contraction on $\dom(G^{m-1})$ for all
  $\ep \in [0,1]$, and by Proposition~\ref{prop:wpc} there exists a
  unique mild solution $u_\ep \in \sC^p(\dom(G^m))$ to
  \eqref{eq:ep}. The latter implies that
  $Gu_\ep \in L^p(\Omega;L^1(0,T;\dom(G^{m-1})))$. Moreover, since
  $f'(u_\ep)$ and $B'(u_\ep)$ satisfy assumption $\mathrm{H}(m-1,p)$,
  the claim follows by Proposition~\ref{prop:wpc}.
\end{proof}

Now are now going to establish well-posedness for \eqref{eq:uen}. We
need further differentiability assumptions on $f$ and $B$, that will
be in force throughout the rest of this section.
\begin{hyp}
  \label{h:dh}
  For each $j \in \{1,m-1\}$, the functions
  \begin{align*}
    f^{(j)}(\omega,t,\cdot) \colon H
    &\longto \dom(G^{m-j-1}),\\
    B^{(j)}(\omega,t,\cdot) \colon H
    &\longto \cL_j(H;\cL^2(U;\dom(G^{m-j-1})))
  \end{align*}
  are of class $C^1_b$ for every $(\omega,t) \in \Omega \times [0,T]$,
  with derivative bounded uniformly with respect to $(\omega,t)$.
\end{hyp}
\noindent Note that, as a consequence of this assumption,
\[
  B^{(m)} (\omega,t,\cdot)\colon H \longrightarrow \cL_m(H;\cL^2(U;H))
\]
is continuous and bounded uniformly with respect to $(\omega,t)$.

\begin{thm}
  Let $n \in \{1,\ldots,m\}$ and $p>0$. If $u_0 \in
  L^p(\Omega,\mathscr{F}_0;\dom(G^m))$, then \eqref{eq:uen} admits a
  unique mild solution $u_\ep^n \in \sC^{p/n}(\dom(G^{m-n})$.
\end{thm}
\begin{proof}
  Recall that $S_{A+\ep G}$ is a strongly continuous semigroup of
  quasi-contractions on $\dom(G^k)$ for every $k \in \{0,\ldots,m\}$ for
  all $\ep \in [0,1]$, and $u_\ep \in \sC^p(\dom(G^m))$. Moreover,
  Proposition~\ref{prop:ufo} yields that the claim is true for
  $n=1$. Therefore it suffices to show that if the claim is true for
  all $n \in \{1,\ldots,n'-1\}$, then it is true also for $n=n'$. To
  this purpose, note that $f'$ and $B'$ satisfy assumption
  $\mathrm{H}(m-n,p)$ for every $n \in \{1,\ldots,m\}$. Moreover,
  each term of $\Phi_n(u_\ep)$ is of the type
  \[
    B^{(j)}(u_\ep)(u_\ep^{n_1},\ldots,u_\ep^{n_j}), \qquad
    n_1+\cdots+n_j=n, \quad j \leq n,
  \]
  hence the differentiability assumptions on $B$ imply
  \begin{align*}
  &\norm[\big]{B^{(j)}(u_\ep)(u_\ep^{n_1},\ldots,%
    u_\ep^{n_j})}_{\cL^2(U;\dom(G^{m-j-1}))}\\
  &\hspace{3em} \leq
    \norm[\big]{B^{(j)}(u_\ep)}_{\cL_j(H;\cL^2(U;\dom(G^{m-j-1})))}
    \norm[\big]{u_\ep^{n_1}} \, \cdots \,
    \norm[\big]{u_\ep^{n_j}}.
  \end{align*}
  Since $\dom(G^{m-n})$ is contractively embedded in $\dom(G^{m-j-1})$
  for every $j \in \{2,\ldots,n\}$, we have, by H\"older's inequality,
  \begin{align*}
  &\norm[\big]{B^{(j)}(u_\ep)(u_\ep^{n_1},\ldots,%
    u_\ep^{n_j})}_{L^{p/n}(\Omega;L^2(0,T;\cL^2(U;\dom(G^{m-n}))))}\\
  &\hspace{3em} \lesssim \sqrt{T} \norm[\Big]{%
    \norm[\big]{u_\ep^{n_1}}_{C([0,T];H)} \, \cdots \,
    \norm[\big]{u_\ep^{n_j}}_{C([0,T];H)}}_{L^{p/n}(\Omega)}\\
  &\hspace{3em} \lesssim_T \norm[\big]{u_\ep^{n_1}}_{\sC^{p/n_1}}
    \, \cdots \, \norm[\big]{u_\ep^{n_j}}_{\sC^{p/n_j}}.
  \end{align*}
  where
  \[
  \frac{n}{p} = \frac{n_1}{p} + \cdots + \frac{n_j}{p}.
  \]
  By the inductive reasoning mentioned above, since obviously
  $n_1,\ldots,n_j < n$ for every $j \leq n$, we conclude that indeed
  $\Phi_n(u_\ep) \in
  L^{p/n}(\Omega;L^2(0,T;\cL^2(U;\dom(G^{m-n}))))$. Applying an entirely
  similar reasoning to $\varphi_n(u_\ep)$ we infer that the maps $x
  \mapsto f'(u_\ep)x + \varphi_n(u_\ep)$ and $x \mapsto B'(u_\ep)x +
  \Phi_n(u_\ep)$ satisfy assumption $\mathrm{H}(m-n,p/n)$, hence the
  proof is completed invoking Proposition~\ref{prop:wpc}.
\end{proof}

\section{First-order differentiability}
\label{sec:d1}
Let Assumption~\ref{h:d1} as well as the hypotheses of
Proposition~\ref{prop:ufo} be in force throughout this section.
\begin{thm}
  The map $\ep \mapsto u_\ep$ is of class $C^1$ from $[0,1]$ to
  $\sC^p(\dom(G^{m-1}))$ and $u'_\ep = u^1_\ep$ for all $\ep \in
  [0,1]$.
\end{thm}
\begin{proof}
  Let us set, for any $\ep \in [0,1]$ and $h>0$,
  \[
    w_{\ep,h} := \frac{u_{\ep+h}-u_\ep}{h} - u^1_\ep.
  \]
  We are going to show that $w_{\ep,h}$ converges to zero in
  $\sC^p(\dom(G^{m-1})$ as $h \to 0$. For the sake of simplicity we
  shall assume that $f = 0$, as the calculations regarding this term
  are entirely similar to those for $B$, and in fact quite simpler.
  Elementary manipulations involving the Duhamel formula show that
  $w_{\ep,h}$ is the unique mild solution to the equation
  \begin{align*}
    &dw_{\ep,h} + (A+\ep G)w_{\ep,h}\,dt + G(u_{\ep+h}-u_\ep) = 
      \Bigl( \frac{B(u_{\ep+h})-B(u_\ep)}{h} - B'(u_\ep)u^1_\ep \Bigr)\,dW,\\
    &w_{\ep,h}(0) = 0,
  \end{align*}
  where, by the mean value theorem,
  \[
    B(u_{\ep+h}) - B(u_\ep) = \int_0^1 B'(u_\ep +
    \theta(u_{\ep+h}-u_\ep)) (u_{\ep+h}-u_\ep)\,d\theta,
  \]
  hence also
  \begin{align*}
  &\frac{B(u_{\ep+h})-B(u_\ep)}{h} - B'(u_\ep)u^1_\ep\\
  &\hspace{3em} = \int_0^1 \Bigl( B'(u_\ep + \theta(u_{\ep+h}-u_\ep)) 
    \frac{u_{\ep+h}-u_\ep}{h} - B'(u_\ep)u^1_\ep \Bigr)\,d\theta\\
  &\hspace{3em} = \int_0^1 B'(u_\ep + \theta(u_{\ep+h}-u_\ep))
    w_{\ep,h}\,d\theta\\
  &\hspace{3em} \quad + \int_0^1 \bigl( B'(u_\ep + \theta(u_{\ep+h}-u_\ep))
    - B'(u_\ep) \bigr) u^1_\ep\,d\theta.
  \end{align*}
  By Proposition~\ref{prop:wpc} we have that $G(u_{\ep+h}-u_\ep) \to
  0$ in $L^p(\Omega;L^1(0,T;\dom(G^{m-1})))$ as $h \to 0$. Moreover,
  since $u_{\ep+h} \to u_\ep$ in $\sC^p$, there exists a sequence $h'$
  such that $\sup_{[0,T]} \abs{u_{\ep+h'}-u_\ep} \to 0$ a.s., hence, by
  continuity of $B'$,
  \[
    \lim_{h'\to 0} B'(u_\ep + \theta(u_{\ep+h'}-u_\ep))x = B'(u_\ep)x
  \]
  for all $\theta \in [0,1]$ and $x \in H$ a.e. in
  $\Omega \times [0,T]$. The dominated convergence theorem then
  implies, by the boundedness of $B'$, that
  \[
    \lim_{h'\to 0} \int_0^1 B'(u_\ep + \theta(u_{\ep+h'}-u_\ep))x\,d\theta
    = B'(u_\ep)x
  \]
  for all $x \in H$ a.e. in $\Omega \times [0,T]$. Similarly,
  \[
  \lim_{h' \to 0} \int_0^1 \bigl( B'(u_\ep + \theta(u_{\ep+h'}-u_\ep))
    - B'(u_\ep) \bigr) u^1_\ep\,d\theta = 0
  \]
  a.e. in $\Omega \times [0,T]$, as well as in
  $L^p(\Omega;L^2(0,T;\cL^2(U;\dom(G^{m-1}))))$ again by dominated
  convergence. Therefore $w_{\ep,h'}$ converges in
  $\sC^p(\dom(G^{m-1}))$ to the unique mild solution $v$ to the
  equation
  \[
    dv + (A + \ep G)v\,dt = B'(u_\ep)v\,dW, \qquad v(0)=0.
  \]
  By linearity it immediately follows that $v=0$. The proof is
  completed observing that from every sequence $h$ converging to zero
  we can extract a subsequence $h'$ such that $w_{\ep,h'} \to 0$,
  hence $w_{\ep,h} \to 0$ as $h \to 0$.
\end{proof}

\section{Higher-order differentiability}
\label{sec:dh}
Let us assume that Assumptions~\ref{h:d1} and \ref{h:dh}, and the
hypotheses of Proposition~\ref{prop:ufo} are fulfilled.
\begin{thm}
  The map $\ep \mapsto u_\ep$ is of class $C^k$ from $[0,1]$ to
  $\sC^{p/k}(\dom(G^{m-k}))$ and $u^{(k)}_\ep = u^k_\ep$ for all $k
  \in \{1,\ldots,m\}$ and $\ep \in [0,1]$.
\end{thm}
\begin{proof}
  Let $\ep \in [0,1]$. We have already proved that $u'_\ep=u^1_\ep$
  and that $u^k_\ep \in C^{p/k}(\dom(G^{m-k}))$ for every $k \in
  \{1,\ldots,m\}$, hence we can achieve the proof by induction as
  follows: assuming that $u^j_\ep=u^{(j)}_\ep$ for all $j \in
  \{1,\ldots,k\}$, we are going to show that
  $u_\ep^{k+1}=u_\ep^{(k+1)}$. As before, we shall assume for
  simplicity that $f=0$. One has
  \begin{gather*}
  du^k_{\ep+h} + (A+\ep G)u^k_{\ep+h}\,dt + hGu^k_{\ep+h}\,dt
  + kGu_{\ep+h}^{k-1}\,dt
  = \bigl( B'(u_{\ep+h})u^k_{\ep+h} + \Phi_k(u_{\ep+h}) \bigr)\,dW,\\
  du^k_\ep + (A+\ep G)u_\ep^k\,dt
  + kGu_\ep^{k-1}\,dt
  = \bigl( B'(u_\ep)u_\ep^k + \Phi_k(u_\ep) \bigr)\,dW,\\
  du^{k+1}_\ep + (A+\ep G)u_\ep^{k+1}\,dt
  + (k+1)Gu_\ep^k\,dt
  = \bigl( B'(u_\ep)u_\ep^{k+1} + \Phi_{k+1}(u_\ep) \bigr)\,dW,
  \end{gather*}
  hence
  \[
  w^k_{\ep,h} := \frac{u^k_{\ep+h}-u^k_\ep}{h} - u^{k+1}_\ep
  \]
  is the unique mild solution to
  \begin{align*}
  &dw^k_{\ep,h} + (A+\ep G)w^k_{\ep,h}\,dt
  + G\bigl(u^k_{\ep+h} - u^k_\ep\bigr)\,dt 
    + kGw^{k-1}_{\ep,h}\,dt\\
  &\hspace{3em} = \Bigl( \frac{B'(u_{\ep+h})u_{\ep+h}^k %
    - B'(u_\ep)u_\ep^k}{h} - B'(u_\ep)u^{k+1}_\ep \Bigr) \,dW\\
  &\hspace{3em} \quad + \Bigl(
    \frac{\Phi_k(u_{\ep+h}) - \Phi_k(u_\ep)}{h} - \Phi_{k+1}(u_\ep)
    \Bigr) \,dW,\\
  &w^k_{\ep,h}(0)=0.
  \end{align*}
  Writing
  \[
  B'(u_{\ep+h})u_{\ep+h}^k - B'(u_\ep)u_\ep^k = B'(u_\ep) \bigl(
  u_{\ep+h}^k - u_\ep^k \bigr) + \bigl( B'(u_{\ep+h}) - B'(u_\ep)
  \bigr) u_{\ep+h}^k
  \]
  yields
  \begin{align*}
  &\frac{B'(u_{\ep+h})u_{\ep+h}^k - B'(u_\ep)u_\ep^k}{h} - B'(u_\ep)u^{k+1}_\ep\\
  &\hspace{3em} = B'(u_\ep) w^k_{\ep,h} %
  + \frac{\bigl( B'(u_{\ep+h}) - B'(u_\ep) \bigr) u_{\ep+h}^k}{h},
  \end{align*}
  thus also
  \begin{align*}
  &dw^k_{\ep,h} + (A+\ep G)w^k_{\ep,h}\,dt
  + G\bigl(u^k_{\ep+h} - u^k_\ep\bigr)\,dt 
    + kGw^{k-1}_{\ep,h}\,dt\\
  &\hspace{3em} = B'(u_\ep) w^k_{\ep,h} \,dW %
    + \Bigl( \frac{\Phi_k(u_{\ep+h}) - \Phi_k(u_\ep)}{h} - \Phi_{k+1}(u_\ep)\\
  &\hspace{5em} + \frac{\bigl( B'(u_{\ep+h}) - B'(u_\ep) \bigr) u_{\ep+h}^k}{h}
    \Bigr) \,dW.
  \end{align*}
  The identities
  \begin{align*}
  B(u_\ep)^{(k)}
  &= B'(u_\ep) u_\ep^{(k)} + \Phi_k(u_\ep),\\
  B(u_\ep)^{(k+1)}
  &= B'(u_\ep) u_\ep^{(k+1)} + \Phi_{k+1}(u_\ep)\\
  &= \bigl( B'(u_\ep) u_\ep^{(k)} \bigr)'
    + \bigl( \Phi_k(u_\ep) \bigr)'\\
  &= B'(u_\ep) u_\ep^{(k+1)} + B''(u_\ep)(u_\ep^k,u_\ep^1)
    + \bigl( \Phi_k(u_\ep) \bigr)'
  \end{align*}
  imply
  \[
  \Phi_{k+1}(u_\ep) = B''(u_\ep)(u_\ep^k,u_\ep^1) %
  + \bigl( \Phi_k(u_\ep) \bigr)',
  \]
  from which it follows that
  \[
  \frac{\Phi_k(u_{\ep+h}) - \Phi_k(u_\ep)}{h} - \Phi_{k+1}(u_\ep)
  + \frac{\bigl( B'(u_{\ep+h}) - B'(u_\ep) \bigr) u_{\ep+h}^k}{h}
  \longto 0
  \]
  in $\cL^2(U;H)$ as $h \to 0$ for all $(\omega,t) \in \Omega \times
  [0,T]$. Moreover, it follows by the inductive assumption that, as $h
  \to 0$, $u^k_{\ep+h} \to u^k_\ep$ and $w^{k-1}_{\ep,h} \to 0$ in
  $\sC^{p/k}(\dom(G^{m-k}))$, which in turn implies
  that
  \[
  G( u^k_{\ep+h} - u^k_\ep ) \longto 0 \quad \text{and} \quad 
  Gw^{k-1}_{\ep,h} \longto 0
  \]
  in $\sC^{p/k}(\dom(G^{m-k-1})) \embed \sC^{p/(k+1)}(\dom(G^{m-k-1}))$.
  Therefore $w^k_{\ep,h}$ converges in
  $\sC^{p/(k+1)}(\dom(G^{m-k-1}))$ to the unique mild solution to the
  equation
  \[
  dv + (A + \ep G)v\,dt = B'(u_\ep)v\,dW, \qquad v(0)=0,
  \]
  which is zero by linearity.
\end{proof}

\bibliographystyle{amsplain}
\bibliography{ref}

\def\polhk#1{\setbox0=\hbox{#1}{\ooalign{\hidewidth
  \lower1.5ex\hbox{`}\hidewidth\crcr\unhbox0}}}
\providecommand{\bysame}{\leavevmode\hbox to3em{\hrulefill}\thinspace}
\providecommand{\MR}{\relax\ifhmode\unskip\space\fi MR }
\providecommand{\MRhref}[2]{%
  \href{http://www.ams.org/mathscinet-getitem?mr=#1}{#2}
}
\providecommand{\href}[2]{#2}
\begin{thebibliography}{10}

\bibitem{cm:epsd1}
S.~Albeverio, C.~Marinelli, and E.~Mastrogiacomo, \emph{Singular perturbations
  and asymptotic expansions for {SPDEs} with an application to term structure
  models}, 2020, arXiv preprint.

\bibitem{Bog:TVS}
V.~I. Bogachev and O.~G. Smolyanov, \emph{Topological vector spaces and their
  applications}, Springer, Cham, 2017. \MR{3616849}

\bibitem{BuBe}
P.~L. Butzer and H.~Berens, \emph{Semi-groups of operators and approximation},
  Springer Verlag New York Inc., New York, 1967. \MR{0230022}

\bibitem{DP:K}
G.~Da~Prato, \emph{Kolmogorov equations for stochastic {PDE}s}, Birkh\"auser
  Verlag, Basel, 2004. \MR{2111320 (2005m:60002)}

\bibitem{EnNa}
K.-J. Engel and R.~Nagel, \emph{One-parameter semigroups for linear evolution
  equations}, Springer Verlag, New York, 2000. \MR{MR1721989 (2000i:47075)}

\bibitem{Kato:sing}
T.~Kato, \emph{Singular perturbation and semigroup theory}, Turbulence and
  {N}avier-{S}tokes equations ({P}roc. {C}onf., {U}niv. {P}aris-{S}ud, {O}rsay,
  1975), Lecture Notes in Math., vol. 565, 1976, pp.~104--112. \MR{0458244}

\bibitem{Kato}
\bysame, \emph{Perturbation theory for linear operators}, Springer Verlag,
  Berlin, 1995, Reprint of the 1980 edition. \MR{1335452}

\bibitem{KvN2}
M.~Kunze and J.~van Neerven, \emph{Continuous dependence on the coefficients
  and global existence for stochastic reaction diffusion equations}, J.
  Differential Equations \textbf{253} (2012), no.~3, 1036--1068. \MR{2922662}

\bibitem{cm:SIMA18}
C.~Marinelli, \emph{On well-posedness of semilinear stochastic evolution
  equations on {$L_p$} spaces}, SIAM J. Math. Anal. \textbf{50} (2018), no.~2,
  2111--2143. \MR{3784905}

\bibitem{cm:JFA13}
C.~Marinelli, G.~Ziglio, and L.~Di~Persio, \emph{Approximation and convergence
  of solutions to semilinear stochastic evolution equations with jumps}, J.
  Funct. Anal. \textbf{264} (2013), no.~12, 2784--2816. \MR{3045642}

\end{thebibliography}

\end{document}